\documentclass[a4paper]{article}
\usepackage[a4paper,top=3cm,bottom=2cm,left=3cm,right=3cm,marginparwidth=1.75cm]{geometry}

\usepackage{makecell}
\usepackage{fixmath}
\usepackage{arydshln}
\usepackage{pdflscape}
\usepackage{arydshln}
\usepackage{amsmath}
\usepackage{amssymb}
\usepackage{multirow}
\usepackage{mathtools}
\usepackage{amsthm}
\usepackage{graphicx}
\usepackage[colorinlistoftodos]{todonotes}
\usepackage[colorlinks=true, allcolors=blue]{hyperref}
\DeclareMathOperator{\PG}{PG}
\DeclareMathOperator{\im}{Im}
\newcommand{\qbin}[2]{\genfrac{[}{]}{0pt}{}{#1}{#2}}

\usepackage[english]{babel}
\usepackage[utf8x]{inputenc}
\usepackage[T1]{fontenc}

\newtheorem{definition}{Definition}[section]
\newtheorem{theorem}[definition]{Theorem}
\newtheorem{lemma}[definition]{Lemma}
\newtheorem{gevolg}[definition]{Corollary}
\newtheorem{remark}[definition]{Remark}
\newtheorem{example}[definition]{Example}

\renewenvironment{proof}[1][\noindent Proof]{{\par\pushQED{\qed}\itshape #1\@. }}{\popQED}

\title{Equivalent definitions for (degree one) Cameron-Liebler classes of generators in finite classical polar spaces}
\author{M. De Boeck, J. D'haeseleer}
\date{}
\begin{document}
\maketitle

\begin{abstract}
In this article, we study \emph{degree one Cameron-Liebler} sets of generators in all finite classical polar spaces, which is a particular type of a Cameron-Liebler set of generators in this polar space, \cite{CLpolar}. These degree one Cameron-Liebler sets are defined similar to the Boolean degree one functions, \cite{Ferdinand.}. We summarize the equivalent definitions for these sets and give a classification result for the degree one Cameron-Liebler sets in the polar spaces $W(5,q)$ and $Q(6,q)$.

\end{abstract}

\section{Introduction}

The investigation of Cameron-Liebler sets of generators in polar spaces is inspired by the research on Cameron-Liebler sets in finite projective spaces. This research started with Cameron-Liebler sets of lines in $\PG(3,q)$, defined by P. Cameron and R. Liebler in \cite{begin}. A set $\mathcal{L} $ of lines in $\PG(3,q)$ is a Cameron-Liebler set of lines if and only if the number of lines in $\mathcal{L}$ disjoint to a given line $l$ only depends on whether $l  \in \mathcal{L}$ or not.
We also find Cameron-Liebler sets in the theory of tight sets for graphs: Cameron-Liebler sets are the tight sets of the type $I$ \cite{bart}.

After many results about those Cameron-Liebler line sets in $\PG(3,q)$, the Cameron-Liebler set concept has been generalized to many other contexts: Cameron-Liebler line sets in $\PG(n,q)$ \cite{phdDrudge}, Cameron-Liebler sets of $k$-spaces in $\PG(2k+1,q)$ \cite{CLkclas}, Cameron-Liebler sets of $k$-spaces in $\PG(n,q)$ \cite{CLksetn}, Cameron-Liebler classes in finite sets \cite{CLset,eenextra,tweeextra} and Cameron-Liebler sets of generators in finite classical polar spaces \cite{CLpolar} were defined. The central problem for Cameron-Liebler sets is to find for which parameter $x$ a Cameron-Liebler set exists, and finding examples with this parameter \cite{phdDrudge,feng,CL20,CL21,Klaus,CL26}. 


In this article, we will investigate Cameron-Liebler sets in finite classical polar spaces. 
The finite classical polar spaces are the hyperbolic quadrics $Q^+(2d-1,q)$, the parabolic quadrics $Q(2d,q)$, the elliptic quadrics $Q^-(2d+1,q)$, the hermitian polar spaces $H(2d-1,q^2)$ and $H(2d,q^2)$, and the symplectic polar spaces $W(2d-1,q)$, with $q$ prime power. 
Here we investigate the sets of generators defined by the following definition, where $A$ is the incidence matrix of points and generators, and we call these sets \emph{degree one Cameron-Liebler sets}. 

\begin{definition}\label{defspecialCL}
A degree one Cameron-Liebler set of generators in a finite classical polar space $\mathcal{P}$ is a set of generators in $\mathcal{P}$, with characteristic vector $\chi$ such that $\chi \in \im(A^T)$. 
\end{definition}

This definition corresponds with the definition of Boolean degree one functions for generators in polar spaces, in \cite{Ferdinand.} by Y. Filmus and F. Ihringer. In their article, they define Boolean degree one functions, or Cameron-Liebler sets in projective and polar spaces by the fact that the corresponding characteristic vector lies in $V_0\perp V_1$, which are eigenspaces of the related association scheme (see Section \ref{section2}).
In \cite{CLpolar}, M. De Boeck, M. Rodgers, L. Storme and A. \v{S}vob introduced Cameron-Liebler sets of generators in the finite classical polar spaces. In this article, Cameron-Liebler set of generators in the polar spaces are defined by the \emph{disjointness-definition} and the authors give several equivalent definitions for these Cameron-Liebler sets.

\begin{definition}[{\cite{CLpolar}}]\label{defCL}
Let $\mathcal{P}$ be a finite classical polar space with parameter $e$ and rank $d$. A set $\mathcal{L}$ of generators in $\mathcal{P}$ is a Cameron-Liebler set of generators in $\mathcal{P}$ if and only if for every generator $\pi$ in $\mathcal{P}$, the number of elements of $\mathcal{L}$, disjoint from $\pi$ equals $(x-\chi(\pi))q^{\binom{d-1}{2}+e(d-1)}$. 
\end{definition}


\begin{table}[h]\begin{center}
  \begin{tabular}{ | c | c| c| }
    \hline
    Type $I$ & Type $II$ & Type $III$ \\ \hline 
    $Q^-(2d+1,q)$ & $Q^+(2d-1,q)$, $d$ even& $Q(4n+2,q)$  \\ 
    $Q(2d,q)$, $d$ even & & $W(4n+1,q)$  \\ 
    $Q^+(2d-1,q)$, $d$ odd & &  \\ 
    $W(2d-1,q)$, $d$ even & & \\ 
    $H(2d-1,q)$, $q$ square & &  \\ 
    $H(2d,q)$, $q$ square & &  \\ \hline
  \end{tabular}
  \caption{Three types of polar spaces}\label{tabeltype}
\end{center}\end{table}

In this article, we consider three different types of polar spaces,  see Table \ref{tabeltype}. Type $I$ and $II$ corresponds with type $I$ and $II$ respectively, defined in \cite{CLpolar}, while type $III$ in this paper corresponds with the union of type $III$ and $IV$ in $\cite{CLpolar}$, as we handle the symplectic polar spaces $W(4n+1,q)$, for both $q$ odd and $q$ even, in the same way.  Definition \ref{defCL} and Definition \ref{defspecialCL} are equivalent for the polar spaces of type $I$ by \cite[Theorem 3.7, Theorem 3.15]{CLpolar}. For the polar spaces of type $II$ we can consider the (degree one) Cameron-Liebler sets of one class of generators; we see that Cameron-Liebler sets and degree one Cameron-Liebler sets coincide when we only consider one class (see \cite[Theorem 3.16]{CLpolar}). For the polar spaces of type $III$, this equivalence no longer applies and for these polar spaces, any degree one Cameron-Liebler set is also a regular Cameron-Liebler set, but not vice versa.


 Cameron-Liebler sets were introduced by a group-theoretical argument: a set $\mathcal{L}$ of lines is a Cameron-Liebler set of lines in $\PG(3,q)$ if and only if PGL$(3,q)$ has the same number of orbits on the lines of $\mathcal{L}$ and on the points of $\PG(3,q)$. 
 
 If the incidence matrix $A$ of points and generators of a polar space $\mathcal{P}$ has trivial kernel, then we also find a group-theoretical definition for degree one Cameron-Liebler sets of generators in $\mathcal{P}$. This theorem follows from {\cite[Lemma 3.3.11]{phdfred}}.
 \begin{theorem}
 Let $X$ be the set of points in a classical polar space $\mathcal{P}$, let $M$ be the set of generators in $\mathcal{P}$ and let $A$ be the point-generator matrix of $\mathcal{P}$. Consider an automorphismgroup $G$ acting on the sets $X$ and $M$ with orbits $O_1, \dots, O_n$ and $O'_1, \dots, O'_m$, respectively. If $A$ has trivial kernel, then each $O'_i$ is a degree one Cameron-Liebler set in $\mathcal{P}$ if and only if $n=m$.
 \end{theorem}
 

In Section $2$ we give some preliminaries about the classical polar spaces and we discuss several properties of the eigenvalues of the association scheme for generators of finite classical polar spaces. In Section $3$, we give an overview of the equivalent definitions and several properties of degree one Cameron-Liebler sets in polar spaces. In Section $4$ we give an equivalent definition for Cameron-Liebler sets in the hyperbolic quadrics $Q^+(2d-1,q)$, $d$ even and in Section $5$ we end with some classification results for degree one Cameron-Liebler sets, especially in the polar spaces $W(5,q)$ and $Q(6,q)$.

\section{Preliminaries}\label{section2}
For an extensive and detailed introduction about distance-regular graphs, polar spaces and association schemes for generators of finite classical polar spaces, we refer to \cite{CLpolar}. For more general information about association schemes of distance regular graphs, we refer to \cite{bose, brouwer}. We only repeat the necessary definitions and information. Note that in this article, we will work in a projective context and all mentioned dimensions are projective dimensions.
\subsection{Finite classical polar spaces}\label{setcionfcps}
We start with the definition of finite classical polar spaces.
\begin{definition}
Finite classical polar spaces are incidence geometries consisting of subspaces that are totally isotropic with respect to a non-degenerate quadratic or non-degenerate reflexive sesquilinear form on a vector space $\mathbb{F}_q^{n+1}$. 
\end{definition}
In this article all polar spaces we will handle are the finite classical polar spaces, so we will call them the polar spaces.
We also give the definition of the rank and the parameter $e$ of a polar space.
\begin{definition}
A generator of a polar space is a subspace of maximal dimension and the rank $d$ of a polar space is the projective dimension of a generator plus $1$. The parameter of a polar space $\mathcal{P}$ over $\mathbb{F}_q$ is defined as the number $e$ such that the number of generators through a $(d-2)$-space of $\mathcal{P}$ equals $q^e+1$.
\end{definition}
In Table \ref{tabele} we give the parameter $e$ of the polar spaces. 

\begin{table}[h]\begin{center}
  \begin{tabular}{ | c | c| }
    \hline
    Polar space & $e$  \\ \hline \hline
    $Q^+(2d-1,q)$ & $0$  \\ \hline
    $H(2d-1,q)$ & $1/2$  \\ \hline
    $W(2d-1,q)$ & $1$  \\ \hline
    $Q(2d,q)$ & $1$  \\ \hline
    $H(2d,q)$ & $3/2$  \\ \hline
    $Q^-(2d+1,q)$ & $2$  \\ \hline
  \end{tabular}
  \caption{The parameter of the polar spaces}\label{tabele}
\end{center}\end{table}

To ease the notations, we will work with the \emph{Gaussian binomial coefficient} $\begin{bmatrix}a\\b\end{bmatrix}_q$ for positive integers $a,b$ and prime power $q\geq 2$:
\begin{align*}
\begin{bmatrix}a\\b\end{bmatrix}_q=\prod_{i=1}^b \frac{q^{a-b+i}-1}{q^i-1} = \frac{(q^a-1)\dots (q^{a-b+1}-1)}{(q^b-1)\dots (q-1)}.
\end{align*}
We will write $\qbin ab$ if the field size $q$ is clear from the context. The number $\qbin ab_q$ equals the number of $(b-1)$-spaces in $\PG(a-1,q)$, and the equality $\qbin ab = \qbin{a}{a-b}$ follows immediately from duality.\\

We end this section by defining several substructures in polar spaces. A \emph{spread} in a polar space $\mathcal{P}$ is a set $S$ of generators in $\mathcal{P}$ such that every point of $\mathcal{P}$ is contained in precisely one element of $S$. A \emph{point-pencil} in a polar space $\mathcal{P}$ with vertex $P\in \mathcal{P}$ is the set of all generators in $\mathcal{P}$ through the point $P$. \\

\subsection{Eigenvalues of the association scheme for generators in polar spaces}
First of all, remark that in this article vectors are regarded as column vectors and we define $\textbf{\textit{j}}_n$ to be the all one vector of length $n$. We write $\textbf{\textit{j}}$ for $\textbf{\textit{j}}_n$ if the length is clear from the context.\\
Let $\mathcal{P}$ be a finite classical polar space of rank $d$ and let $\Omega$ be its set of generators. The relations $R_i$ on $\Omega$ are defined as follows: $(\pi,\pi') \in R_i$ if and only if $\dim(\pi \cap \pi') =d-i-1$, for generators $\pi,\pi'\in \Omega$ with $i=0, ...,d$. We define $A_i$ as the incidence matrix of the relation $R_i$. By the theory of association schemes we know that there is an orthogonal decomposition $V_0 \perp V_1 \perp \cdots \perp V_d$ of $\mathbb{R}^\Omega$ in common eigenspaces of $A_0,A_1,...,A_d$.

\begin{lemma}[{\cite[Theorem 4.3.6]{phdfred}}]\label{eigenvallem}
In the association scheme of a polar space over $\mathbb{F}_q$ of rank $d$ and parameter $e$, the eigenvalue $P_{ji}$ of the relation $R_i$ corresponding to the subspace $V_j$ is given by:
\begin{align*}
P_{ji} = \sum\limits_{s=\max{(0,j-i)}}^{\min{(j,d-i)}} (-1)^{j+s}\begin{bmatrix} j\\s\end{bmatrix} \begin{bmatrix}d-j \\ d-i-s\end{bmatrix} q^{e(i+s-j)+\binom{j-s}{2}+\binom{i+s-j}{2}}.
\end{align*}.
\end{lemma}
Before we start with investigating the Cameron-Liebler sets of generators in finite classical polar spaces, we give an important lemma about the eigenvalues $P_{ji}$.\\

\begin{lemma} \label{lemma2}
In the association scheme of polar spaces, the eigenvalue $P_{1i}$ of $A_i$ corresponds only with the eigenspace $V_1$  for $i\neq 0$, except in the following cases.
\begin{enumerate}
\item The hyperbolic quadrics $Q^+(2d-1,q)$. Here $P_{1i}=P_{d-1,i}$ for $i$ even,  so $P_{1i}$ also corresponds with $V_{d-1}$, for every relation $R_i$, $i$ even.
\item The parabolic quadrics $Q(4n+2,q)$ and the symplectic quadrics $W(4n+1,q)$. Here $P_{1d} = P_{dd}$, so $P_{1d}$ also corresponds with $V_d$ for the disjointness relation $R_d$.
\end{enumerate}
\end{lemma}
\begin{proof}
We need to prove, given a fixed $i$ and $j\neq 1$, that $P_{1i} \neq P_{ji}$ for $q$ a prime power.  For $j=0$ and for all $i\neq 0$, it is easy to calculate that $P_{1i}\neq P_{0i}$, so we can suppose that $j>1$.

For $i=1$ we can directly compare the eigenvalues $P_{11}$ and $P_{j1}$.
\begin{align*}
P_{11} = P_{j1} &\Leftrightarrow  \begin{bmatrix}d-1 \\ 1\end{bmatrix} q^e -1= \begin{bmatrix}d-j \\ 1\end{bmatrix} q^e -\begin{bmatrix}j \\ 1 \end{bmatrix}\\
& \Leftrightarrow \frac{-q+1+(q^{d-1}-1)q^e}{q-1}=\frac{-q^j+1+(q^{d-j}-1)q^e}{q-1}\\
& \Leftrightarrow (q^{d-j+e-1}+1)(q^{j-1}-1)=0
\end{align*}
Since $j>1$ the last equation gives a contradiction for any $q$. 

For $i\geq 2$ we introduce $\phi_i(j) = \max\{k\mid \mid q^k|P_{ji} \}$, the exponent of
$q$ in $P_{ji}$. If $P_{ij}=0$, we put $\phi_i(j)=\infty$. We will show that $\phi_i(j)$  is different from $\phi_i(1)$ for most values of $i$ and $j$. For $j=1$, we find that 
\begin{align*}
P_{1i}=-\begin{bmatrix}
d-1 \\ d-i\end{bmatrix} q^{\binom{i-1}{2}+e(i-1)}+\begin{bmatrix}
d-1 \\ i\end{bmatrix} q^{\binom{i}{2}+ei}=q^{\binom{i-1}{2}+e(i-1)} \left(\qbin{d-1}{i}q^{i-1+e} -\qbin{d-1}{i-1} \right).
\end{align*}
We can see that $\phi_i(1)=\binom{i-1}{2}+e(i-1)$, since $i-1+e\geq  1$ and $\qbin ab=1 \pmod q$ for all $0\leq b\leq a$.

In Lemma \ref{eigenvallem} we see that ${\phi_i(j)}$ depends on the last factor of every term in the sum. 
To find $\phi_i(j)$ we first need to find $z$ such that $q^{e(i+z-j)+\binom{j-z}{2}+\binom{i+z-j}{2}}$ is a factor of every term in the sum, or equivalently, such that $f_{ji}(s)=e(i+s-j)+\binom{j-s}{2}+\binom{i+s-j}{2}$ reaches its minimum for $s=z$. So for most cases, we have that $\phi_i(j)=f_{ij}(z)$, but in some cases it occurs that two values of $z$ correspond with opposite terms with factor $q^{\phi_i(j)}$. These cases, we have to investigate separately. \\
We can check that $z$ is the integer or integers in $[\max\{0,j-i\}, \dots, \min\{j,d-i\}]$, closest to $j-\frac{i}{2}-\frac{e}{2}$. 
Since $i\geq 2$ we have three possibilities for the value of $z$, as we always have $j-i\leq j-\frac{i}{2}-\frac{e}{2}<j$:
\begin{itemize}
    \item $z=0$ if $j-\frac{i}{2}-\frac{e}{2}<0$,
    \item $z \in \{j-\frac{i}{2}-\frac{e}{2}, j-\frac{i}{2}-\frac{e}{2}\pm \frac{1}{2}\}$ if $0\leq j-\frac{i}{2}-\frac{e}{2}\leq d-i$,
    \item $z=d-i$ if $j-\frac{i}{2}-\frac{e}{2}>d-i$.
\end{itemize}

Now we handle these three cases.
\begin{itemize}
\item If $j-\frac{i}{2}-\frac{e}{2} < 0$, we see that $f_{ji}$ is minimal for the integer $z=0$.

We remark that in this case there is only $1$ value of $s$, namely $0$, for which the corresponding term is divisible by $q^{\phi_i(j)}$ but not by $q^{\phi_i(j)+1}$. This is important to exclude the case where $2$ terms with factor $q^{\phi_i(j)}$ would be each others opposite.

We find that $\phi_i(j)=f_{ji}(0)=\binom{i}{2}+(j-i)(j-e)$, and since $\phi_i(1)=\binom{i-1}{2}+e(i-1)$, the values $\phi_i(j)$ and $\phi_i(1)$ are equal if and only if $j= 1 \vee j= i+e-1$. We only have to check the latter case, and recall that $ j-\frac{i}{2}-\frac{e}{2} < 0$. It follows that $i+e< 2$, a contradiction since we supposed $i\geq 2$.

\item If $0 \leq j-\frac{i}{2}-\frac{e}{2} \leq d-i$ we see that $f_{ji}$ is minimal for the integer $z$ closest to $j-\frac{i}{2}-\frac{e}{2}$.

\def\mystrut{\vrule height 5ex depth 3.5ex width 0pt}
\pagestyle{empty}
\begin{table}[hp]
\centering
\renewcommand{\arraystretch}{2.1}
\setlength\extrarowheight{-2pt}
\begin{tabular}{|>{\centering\arraybackslash}p{6mm}|p{20mm}|p{26mm}|p{45mm}|>{\centering\arraybackslash}p{20mm}|>{\centering\arraybackslash}p{6mm}|}\hline
\boldmath{$e$} & \boldmath{$i$} & \boldmath{$z$} & \boldmath{$\phi_i(j)=f_{ji}(z)$} &\boldmath{$\phi_i(1)$}&  \boldmath{$S$}\\ \hline
\hline \multicolumn{6}{|c|}{$Q^+ (2d-1,q)$}\\ \hline
\multirow{2}{*}{$0$}
  & even &  $j-\frac{i}{2}$  & $\frac{i(i-2)}{4}$&$\frac{(i-1)(i-2)}{2}$ & $\{2\}$\\ \cline{2-6}
 & odd & $j-\frac{i}{2}\pm \frac{1}{2}$  & $\begin{cases} \frac{(i-1)^2}{4} &\mbox{if } j\neq \frac{d}{2} \\ 
 \infty & \mbox{if } j=\frac{d}{2} \end{cases}$&$\frac{(i-1)(i-2)}{2}$ & $\{3\}$\\ \hline
\hline \multicolumn{6}{|c|}{$\mathcal{H}(2d-1,q)$, with $q$ square}\\ \hline
\multirow{2}{*}{$\frac{1}{2}$}
 & even & $j-\frac{i}{2}$ & $\frac{i(i-1)}{4}$&$\frac{(i-1)^2}{2}$ & $\{2\}$\\ \cline{2-6}
 & odd & $j-\frac{i}{2}-\frac{1}{2}$ & $\frac{i(i-1)}{4}$&$\frac{(i-1)^2}{2}$ & $\emptyset$\\ \hline
\hline \multicolumn{6}{|c|}{$Q(2d,q)$, $W(2d-1,q)$, with $d\not\equiv 0 \pod{4}$}\\ \hline
\multirow{2}{*}{$1$}
 & even & $j-\frac{i}{2}-\frac{1}{2}\pm \frac{1}{2}$  & $ \frac{i^2}{4}$&$\frac{i(i-1)}{2}$ & $\{2\}$\\ \cline{2-6}
 & odd & $j-\frac{i}{2}-\frac{1}{2}$ & $\frac{i^2-1}{4}$&$\frac{i(i-1)}{2}$ & $\emptyset$\\ \hline
\hline \multicolumn{6}{|c|}{$Q(2d,q)$, $W(2d-1,q)$, with $d\equiv 0 \pod{4}$}\\ \hline
\multirow{3}{*}{$1$}
& even, $i\neq\frac{d}{2}$ & $j-\frac{i}{2}-\frac{1}{2}\pm \frac{1}{2}$  & $\frac{i^2}{4}$&$\frac{i(i-1)}{2}$ & $\{2\}$\\ \cline{2-6}
 & $i=\frac{d}{2}$ & $j-\frac{i}{2}-\frac{1}{2}\pm \frac{1}{2}$  & \mystrut $\begin{cases} \infty &\text{if $j= \frac{d}{2}+1$}\\ \frac{i^2}{4} & \text{else} \end{cases}$&$\frac{i(i-1)}{2}$ & $\{2\}$\\ \cline{2-6}
 & odd & $j-\frac{i}{2}-\frac{1}{2}$ & $\frac{i^2-1}{4}$&$\frac{i(i-1)}{2}$ & $\emptyset$\\ \hline
\hline \multicolumn{6}{|c|}{$\mathcal{H}(2d,q)$, with $q$ square}\\ \hline
\multirow{2}{*}{$\frac{3}{2}$}
& even & $j-\frac{i}{2}-1$ & $\frac{(i-1)(i+2)}{4}$&$\frac{i^2-1}{2}$ & $\emptyset$\\ \cline{2-6}
 & odd & $j-\frac{i}{2}-\frac{1}{2}$ & $\frac{(i-1)(i+2)}{4}$ &$\frac{i^2-1}{2}$ & $\emptyset$\\ \hline
\hline \multicolumn{6}{|c|}{$Q^- (2d+1,q)$, with $d\not\equiv 2 \pod{4}$}\\ \hline
\multirow{2}{*}{$2$}
 & even &  $j-\frac{i}{2}-1$  & $\frac{i^2}{4}+\frac{i}{2}-1$&$\frac{(i-1)(i+2)}{2}$ & $\emptyset$\\ \cline{2-6}
 & odd & $j-\frac{i}{2}-1\pm\frac{1}{2}$  & $\frac{(i-1)(i+3)}{4}$&$\frac{(i-1)(i+2)}{2}$ & $\emptyset$\\ \hline
\hline \multicolumn{6}{|c|}{$Q^- (2d+1,q)$, with $d\equiv 2 \pod{4}$}\\ \hline
\multirow{3}{*}{$2$}
 & even &  $j-\frac{i}{2}-1$  & $\frac{i^2}{4}+\frac{i}{2}-1$&$\frac{(i-1)(i+2)}{2}$ & $\emptyset$\\ \cline{2-6}
 & odd, $i\neq\frac{d}{2}$ &  $j-\frac{i}{2}-1\pm \frac{1}{2}$  & $\frac{(i-1)(i+3)}{4}$&$\frac{(i-1)(i+2)}{2}$ & $\emptyset$\\ \cline{2-6}
 & $i=\frac{d}{2}$ & $j-\frac{i}{2}-1\pm\frac{1}{2}$  & \mystrut $\begin{cases} \infty &\text{if $j= \frac{d}{2}+2$} \\ \frac{(i-1)(i+3)}{4} & \text{else}  \end{cases}$&$\frac{(i-1)(i+2)}{2}$ & $\emptyset$\\ \hline
\end{tabular}
\caption{For $0\leq j-\frac{i}{2}-\frac{e}{2} \leq d-i$, with $S=\{ i \geq 2 \mid\phi_i(j)=\phi_i(1)\}$. }
\label{tabel}
\end{table}

In Table \ref{tabel} we list the different cases depending on $e$ and the parity of $i$. Note that we have to check, for $e=0, i$ odd, for $e=1, i$ even, and for $e=2, i$ odd, that the two values of $z$ do not correspond with two opposite terms with factor $q^{\phi_i(j)}$. By calculating and taking in account the conditions $0\leq j-\frac{i}{2}-\frac{e}{2}\leq d-i$, we find out that those cases do not correspond with two opposite terms, except in the following cases:
\begin{itemize}
    \item $e=0, j=\frac{d}{2}$ and $i$ odd,
    \item $e=1, j=\frac{d}{2}+1, i=\frac{d}{2}$ and $i$ even,
    \item $e=2,j=\frac{d}{2}+2, i=\frac{d}{2}$ and $i$ odd.
\end{itemize}
In these cases, $P_{ij}=0$, so $\phi_i(j)=\infty\neq \phi_i(1)$.

  Remark that for every $e$, $i$ and $j>1$, $\phi_i(j)=f_{ij}(z)$ is independent of $j$, see the fifth column in Table \ref{tabel}. In the last column we give the values of $i$ for which $\phi_i(j)=\phi_i(1)$. As we supposed $i\geq 2$, we see that we have to check the eigenvalues for $i=2$ if $e\in \{0,\frac{1}{2}, 1\}$ and for $i=3$ if $e=0$ in detail.
\begin{itemize}

\item  Case $i=2$ and $e\in\{0,\frac{1}{2},1\}$:\small {\begin{align*} 
& P_{12} = P_{j2} \\
& \Leftrightarrow -\begin{bmatrix} d-1 \\ 1\end{bmatrix}q^e+ \begin{bmatrix}d-1 \\ 2\end{bmatrix} q^{1+2e} =\begin{bmatrix} j\\2 \end{bmatrix}q -\begin{bmatrix} d-j\\ 1 \end{bmatrix}\begin{bmatrix}j \\ 1 \end{bmatrix}q^e+\begin{bmatrix}d-j \\ 2\end{bmatrix} q^{1+2e} \\
& \Leftrightarrow \left(\qbin{d-1}{2}-\qbin{d-j}{2} \right) q^{2e} + \qbin{d-j-1}{1}\qbin{j-1}{1} q^{e} =\begin{bmatrix} j\\2 \end{bmatrix}.
\end{align*}
For $e=\frac{1}{2}$ and $e=1$, we see that the right and left hand side of the last equation are different modulo $q$. So we can assume $e=0$.
\begin{align*}
& P_{12} = P_{j2}\\
 &\Leftrightarrow  \frac{(q^{d-1}-1)(q^{d-2}-1)}{(q^2-1)(q-1)}-\frac{(q^{d-j}-1)(q^{d-j-1}-1)}{(q^2-1)(q-1)}+\frac{(q^{d-j-1}-1)(q^{j-1}-1)}{(q-1)(q-1)}   = \frac{(q^j-1)(q^{j-1}-1)}{(q^2-1)(q-1)} \\
& \Leftrightarrow q^{2d-3}-q^{2d-2j-1}+q-q^{2j-1}=0 \\
& \Leftrightarrow q(q^{2j-2}-1)(q^{2(d-j-1)}-1)=0 
\end{align*}}
Since $j>1$, we see that $P_{12}=P_{j2}$ if and only if $j=d-1$. This corresponds with the first exception in the lemma with $i=2$.

\item Case $i=3$ and $e=0$.
\begin{align*}
&P_{13} = P_{j3} \\
&\Leftrightarrow -\begin{bmatrix} d-1 \\ 2\end{bmatrix}q+ \begin{bmatrix}d-1 \\ 3\end{bmatrix} q^{3} =-\begin{bmatrix}j \\ 3 \end{bmatrix}q^3+\begin{bmatrix} j\\2 \end{bmatrix}\begin{bmatrix} d-j\\ 1 \end{bmatrix}q -\begin{bmatrix}j \\ 1 \end{bmatrix}\begin{bmatrix} d-j\\ 2 \end{bmatrix}q+\begin{bmatrix}d-j \\ 3\end{bmatrix} q^{3} \\
&\Leftrightarrow -\begin{bmatrix} d-1 \\ 2\end{bmatrix}+ \begin{bmatrix}d-1 \\ 3\end{bmatrix} q^{2} =-\begin{bmatrix}j \\ 3 \end{bmatrix}q^2+\begin{bmatrix} j\\2 \end{bmatrix}\begin{bmatrix} d-j\\ 1 \end{bmatrix} -\begin{bmatrix}j \\ 1 \end{bmatrix}\begin{bmatrix} d-j\\ 2 \end{bmatrix}+\begin{bmatrix}d-j \\ 3\end{bmatrix} q^{2} 
\end{align*}
Since the right and left hand side of the last equation are different modulo $q$, we see that $P_{13} \neq P_{j3}$ for $j>1$. Recall that $\qbin ab =1 \pmod q$.

\end{itemize}


\item If $ j-\frac{i}{2}-\frac{e}{2} > d-i$, we see that $f_{ji}$ is minimal for the integer $z=d-i$. Remark again that there is only one value of $s$ for which the corresponding term is divisible by $q^{\phi_i(j)}$ but not by $q^{\phi_i(j)+1}$. This excludes the case where $2$ terms with factor $q^{\phi_i(j)}$ would be each others opposite.

We find that $\phi_i(j)=f_{ji}(d-i)=(j-e-d+1)(j-d+i-1)+\binom{i-1}{2}+e(i-1)$, and we know that $\phi_i(1)=\binom{i-1}{2}+e(i-1)$. These two values $\phi_i(j)$ and $\phi_i(1)$ are equal if and only if $j = e+d-1$ or $j= d-i+1$. 
\begin{itemize}
\item Suppose $j=d+e-1$. As $j,d \in \mathbb{Z}$, we know that $e \in \mathbb{Z}$. If $e=2$, then $j=d+1 > d$, a contradiction. For $e=1$, we find that $P_{1i}=P_{di}$ if and only if $i=d$ and $d$ odd. This corresponds to the polar spaces $Q(4n+2,q)$ and $ W(4n+1,q)$. For $e=0$ and $j=d-1$, we find that $P_{1i}=P_{d-1,i}$ for $i$ even. This corresponds to the exception for the polar spaces $Q^+(2d-1,q)$ and $i$ even. 
\item Suppose $j=d-i+1$. Since $j-\frac{i}{2}-\frac{e}{2} > d-i$, we know that $i+e<2$, which gives a contradiction as we supposed $i\geq 2$.\qedhere
\end{itemize}
\end{itemize} 
\end{proof}

We continue with well-known theorems that will be useful in the following sections. The first theorem follows from \cite[Theorem 2.14]{CLpolar} which was originally proven in \cite{delsarte}. For the second theorem we add a proof for completeness. The ideas are already present in \cite[Lemma 2]{bamberg} and \cite[Lemma 2.1.3]{phdfred}.
\begin{theorem}\label{lemmaV0V1}
Let $\mathcal{P}$ be a finite classical polar space of rank $d$ and parameter $e$, and let $\Omega$ be the set of all generators of $\mathcal{P}$. Consider the eigenspace decomposition $\mathbb{R}^\Omega =V_0\perp V_1 \perp \dots \perp V_d$ related to the association scheme, and using the classical order. Let $A$ be the point-generator incidence matrix of $\mathcal{P}$, then $\im(A^T) = V_0 \perp V_1$ and $V_0 = \langle j \rangle$.
\end{theorem}


\begin{theorem} \label{stellingalgemeen}
Let $R_i$ be a relation of an association scheme on the set $\Omega$ with adjacency matrix $A_i$ and let $\mathcal{L} \subset \Omega$ be a set, with characteristic vector $\chi$, such that for any $\pi \in \Omega$, we have that 
\begin{align*}
|\{x\in \mathcal{L} | (x,\pi) \in R_i \}| = \begin{cases} \alpha_i \ \text{If } \pi \in \mathcal{L} \\ \beta_i \ \text{If }  \pi \notin \mathcal{L}
\end{cases}
\end{align*}
with $\alpha_i -\beta_i=P$ an eigenvalue of $A_i$ for the eigenspace $V$, then $v_i= \chi + \frac{\beta_i}{P-P_{0i}}\textbf{\textit{j}} \in V$.

\end{theorem}
Remark that the eigenspace $V$ in the previous theorem can be the direct sum of several eigenspaces of the association scheme. Note that an association scheme is not necessary in this theorem, a regular relation suffices. 
\begin{proof}
We show that $v_i = \chi + \frac{\beta_i}{P - P_{0i}}\textbf{\textit{j}}$, with $P=\alpha_i-\beta_i$ is an eigenvector for the matrix $A_i$ with eigenvalue $P$:
 \begin{align*}
 A_i\left(\chi + \frac{\beta_i}{P - P_{0i}}\textbf{\textit{j}}\right) =& \alpha_i \chi + \beta_i (\textbf{\textit{j}}-\chi) + \frac{\beta_i}{P - P_{0i}}P_{0i} \textbf{\textit{j}} \\
 =& P\left(\chi + \frac{\beta_i}{P - P_{0i}}\textbf{\textit{j}}\right).
 \end{align*}\\
So we find that $\chi + \frac{\beta_i}{P - P_{0i}}\textbf{\textit{j}} \in V$.
\end{proof}\\

\section{Degree one Cameron-Liebler sets}
In this section we investigate the degree one Cameron-Liebler sets and give an equivalent definition. 
Recall that for polar spaces of type $I$ Cameron-Liebler sets and degree one Cameron-Liebler coincide.

Using Lemma \ref{lemma2} and Theorem \ref{stellingalgemeen}, we can give a new equivalent definition for these degree one Cameron-Liebler sets of generators in polar spaces. Remark that the following theorem is an extension of Lemma $4.9$ in \cite{CLpolar}.

\begin{theorem} \label{stelling}
Let $\mathcal{P}$ be a finite classical polar space, of rank $d$ with parameter $e$, let $\mathcal{L}$ be a set of generators of $\mathcal{P}$ and $i$ be an integer with $1\leq i\leq d$.  Denote $\frac{|\mathcal{L}|}{\prod_{i=0}^{d-2}(q^{e+i}+1)}$ by $x$.
If $\mathcal{L}$ is a degree one Cameron-Liebler set of generators in $\mathcal{P}$ then the number of elements of $\mathcal{L}$ meeting a generator $\pi$ in a $(d-i-1)$-space equals
 \begin{align}\label{formulelang}
 \left\{ \begin{matrix}
 \left( (x-1) \begin{bmatrix}d-1 \\i-1\end{bmatrix} +q^{i+e-1}\begin{bmatrix}d-1 \\ i\end{bmatrix} \right) q^{\binom{i-1}{2}+ (i-1)e} & \mbox{If } \pi \in \mathcal{L}\\  x \begin{bmatrix} d-1 \\ i-1\end{bmatrix} q^{\binom{i-1}{2}+(i-1)e} & \mbox{If }\pi \notin \mathcal{L}. 
 \end{matrix}\right.
 \end{align}
 If this propery holds for a polar space $\mathcal{P}$ and an integer $i$ such that
 \begin{itemize}
     \item $i$ is odd for $\mathcal{P}=Q^+(2d-1,q)$,
     \item $i\neq d$ for $\mathcal{P}=Q(2d,q)$ or $\mathcal{P}=W(2d-1,q)$ both with $d$ odd or
     \item $i$ is arbitrary otherwise,
 \end{itemize}
 then  $\mathcal{L}$ is a degree one Cameron-Liebler set with parameter $x$.
\end{theorem}
\begin{proof}
Consider first a degree one Cameron-Liebler set $\mathcal{L}$ of generators in the polar space $\mathcal{P}$ with characteristic vector $\chi$. As $\chi \in V_0\perp V_1$, we have $\chi = v+a\textbf{\textit{j}}$ for some $v\in V_1$ and some $a\in \mathbb{R}$. Since $|\mathcal{L}| = \langle j,\chi \rangle = x\prod_{i=0}^{d-2}(q^{i+e}+1)$, we find that $a=\frac{x}{q^{d+e-1}+1}$, hence $\chi = \frac{x}{q^{d+e-1}+1}\textbf{\textit{j}}+v$.
 Recall that the matrix $A_i$ is the incidence matrix of the relation $R_i$, which describes whether the dimension of the intersection of two generators equals $d-i-1$ or not. This implies that the vector $A_i \chi$, on the position corresponding to a generator $\pi$, gives the number of generators in $\mathcal{L}$, meeting $\pi$ in a $(d-i-1)$-space. We have
\begin{align*}
A_i \chi =& A_iv+\frac{x}{q^{d+e-1}+1}A_i \textbf{\textit{j}}= P_{1i}v+\frac{x}{q^{d+e-1}+1}P_{0i}\textbf{\textit{j}} \\
=& \left( \begin{bmatrix}d-1 \\i \end{bmatrix}q^{\binom{i}{2}+ei}-\begin{bmatrix}d-1 \\ i-1\end{bmatrix}q^{\binom{i-1}{2}+e(i-1)} \right) v+  \frac{x}{q^{d+e-1}+1}\begin{bmatrix}d \\ i\end{bmatrix}q^{\binom{i}{2}+ei} \textbf{\textit{j}}\\
=& \left( \begin{bmatrix}d-1 \\i \end{bmatrix}q^{\binom{i}{2}+ei}-\begin{bmatrix}d-1 \\ i-1\end{bmatrix}q^{\binom{i-1}{2}+e(i-1)} \right) \left(\chi - \frac{x}{q^{d+e-1}+1}\textbf{\textit{j}} \right)+  \frac{x}{q^{d+e-1}+1}\begin{bmatrix}d \\ i\end{bmatrix}q^{\binom{i}{2}+ei} \textbf{\textit{j}}\\
=&\frac{xq^{\binom{i-1}{2}+e(i-1)}}{q^{d+e-1}+1}\left(\begin{bmatrix}d-1 \\i-1 \end{bmatrix}-\begin{bmatrix}d-1 \\i \end{bmatrix}q^{i+e-1}+\begin{bmatrix}d \\i \end{bmatrix}q^{i+e-1}  \right)\textbf{\textit{j}} \\ & + q^{\binom{i-1}{2}+e(i-1)}\left(\begin{bmatrix}d-1 \\i \end{bmatrix}q^{i+e-1}-\begin{bmatrix}d-1 \\i-1 \end{bmatrix}  \right) \chi \\
=& q^{\binom{i-1}{2}+e(i-1)} \left(x\begin{bmatrix}d-1 \\i-1 \end{bmatrix}\textbf{\textit{j}} + \left(\begin{bmatrix}d-1 \\i \end{bmatrix}q^{i+e-1} -\begin{bmatrix}d-1 \\i-1 \end{bmatrix}\right) \chi \right),
\end{align*}
which proves the first implication.\\
For the proof of the other implication, suppose that $\mathcal{L}$ is a set of generators in $\mathcal{P}$ with the property described in the statement of the theorem. We apply Theorem \ref{stellingalgemeen} with $\Omega$ the set of all generators in $\mathcal{P}$, $R_i$ the relation $\{(\pi, \pi')|\dim(\pi \cap\pi') = d-i-1\}$, and 
\begin{align*}
 \alpha_i &= \left( (x-1) \begin{bmatrix}d-1 \\i-1\end{bmatrix} +q^{i+e-1}\begin{bmatrix}d-1 \\ i\end{bmatrix} \right) q^{\binom{i-1}{2}+ (i-1)e},\\
 \beta_i &= x \begin{bmatrix} d-1 \\ i-1\end{bmatrix} q^{\binom{i-1}{2}+(i-1)e}.\end{align*}
 
As $\alpha_i - \beta_i = P_{1i}$, we find that $v_i = \chi + \frac{\beta_i}{P_{1i} - P_{0i}}\textbf{\textit{j}} \in V_1$, for the admissible values of $i$, by Lemma \ref{lemma2}. Hence by Definition \ref{defspecialCL},  $\mathcal{L}$ is a degree one Cameron-Liebler set in $\mathcal{P}$. 
\end{proof}\\

Remark that this definition is also a new equivalent definition for Cameron-Liebler sets of generators in polar spaces of type $I$, as for these polar spaces, degree one Cameron-Liebler sets and Cameron-Liebler sets coincide.

In the following lemma, we give some properties of degree one Cameron-Liebler sets in a polar space.
 
\begin{lemma}\label{lemmapropdegree oneCL}
Let $\mathcal{L}$ be a degree one Cameron-Liebler set of generators in a polar space $\mathcal{P}$ and let $\chi$ be the characteristic vector of $\mathcal{L}$. Denote $\frac{|\mathcal{L}|}{\prod_{i=0}^{d-2}(q^{e+i}+1)}$ again by $x$. Then $\mathcal{L}$ has the following properties:
\begin{enumerate}
\item $\chi = \frac{x}{q^{d+e-1}+1}\textbf{\textit{j}} +v $ with $v \in V_1$,
\item $\chi - \frac{x}{q^{d+e-1}+1}\textbf{\textit{j}}$ is an eigenvector with eigenvalue $P_{1i}$ for all adjacency matrices $A_i$ in the association scheme, 
\item if $\mathcal{P}$ admits a spread, then $|\mathcal{L}\cap S|=x$ for every spread $\mathcal{S}$ of $\mathcal{P}$.
\end{enumerate}
\end{lemma}
\begin{proof}
The first property follows from the first part of the proof of Theorem \ref{stelling}. The second property follows from the first property since $\chi - \frac{1}{q^{d+e-1}+1}\textbf{\textit{j}}\in V_1$.
Consider now a spread $S$ in $\mathcal{P}$ with characteristic vector $\chi_S$ and let $A$ be the point-generator incidence matrix of $\mathcal{P}$. Since $\chi\in \im(A^T) = \ker(A)^\perp$ and by \cite[Lemma 3.6(i), $m=1$]{CLpolar}, which gives that $u=\chi_S-\frac{1}{\prod_{i=0}^{d-2}(q^{e+i}+1)}\textbf{\textit{j}}\in \ker(A)$, we find, by taking the inner product of $u$ and $\chi$, that
\begin{align*}
|\mathcal{L}\cap S| = \langle \chi_S, \chi \rangle = \frac{1}{\prod_{i=0}^{d-2}(q^{e+i}+1)}\langle \textbf{\textit{j}},\chi \rangle = \frac{1}{\prod_{i=0}^{d-2}(q^{e+i}+1)} |\mathcal{L}|= x.
\end{align*}\qedhere \end{proof}

We also give some properties of degree one Cameron-Liebler sets of generators in polar spaces that can easily be proved.
\begin{lemma} \label{basislemma4}
Let $\mathcal{L}$ and $\mathcal{L}'$ be two degree one Cameron-Liebler sets of generators in a polar space $\mathcal{P}$ with parameters $x$ and $x'$ respectively, then the following statements are valid.
\begin{enumerate}
\item $0 \leq x,x' \leq q^{d-1+e}+1$.
\item  $|\mathcal{L}|=x\prod_{i=0}^{d-2}(q^{i+e}+1)$. 
\item The set of all generators in the polar space $\mathcal{P}$ not in $\mathcal{L}$ is a degree one Cameron-Liebler set of generators in $\mathcal{P}$ with parameter $q^{d-1+e}+1-x$.
\item If $\mathcal{L} \cap \mathcal{L}' = \emptyset$ then $\mathcal{L} \cup \mathcal{L}'$ is a degree one Cameron-Liebler set of generators in $\mathcal{P}$ with parameter $x+x'$.
\item  If $\mathcal{L} \subseteq \mathcal{L}'$ then $\mathcal{L} \setminus \mathcal{L}'$ is a degree one Cameron-Liebler set of generators in $\mathcal{P}$ with parameter $x-x'$.
\end{enumerate}
\end{lemma}
\begin{lemma}[{\cite[Lemma 2.3]{Ferdinand.}}]\label{lemmaferdi}
Let $\mathcal{P}$ be a polar space of rank $d$ and let $\mathcal{P'}$ be a polar space, embedded in $\mathcal{P}$ with the same rank $d$. If $\mathcal{L}$ is a degree one Cameron-Liebler set in $\mathcal{P}$, then the restriction of $\mathcal{L}$ to $\mathcal{P'}$ is again a degree one Cameron-Liebler set. 
\end{lemma}
Note that Theorem \ref{stelling} does not hold for some values of $i$, dependent on the polar space $\mathcal{P}$, since for these cases, we cannot apply Lemma \ref{lemma2}. We will now show that there are examples of generator sets that admit the property of Theorem \ref{stelling} for the non-admitted values of $i$, but that are not degree one Cameron-Liebler sets. These are however Cameron-Liebler sets in the sense of \cite{CLpolar}.

\begin{remark}\label{commentvb4.6}
By investigating \cite[Example $4.6$]{CLpolar}, we find an example of a Cameron-Liebler set in a polar space of type $III$ with $d=3$, that is not a degree one Cameron-Liebler set: a base-plane. A \emph{base-plane} in a polar space $\mathcal{P}$ of rank $3$ with base the plane $\pi$ is the set of all planes in $\mathcal{P}$, intersecting $\pi$ in at least a line. 

Let $\mathcal{P}$ be a polar space of type $III$ of rank $3$, so $\mathcal{P}=W(5,q)$ or $\mathcal{P}=Q(6,q)$. Let $\pi$ be a plane and let $\mathcal{L}$ be the base-plane with base $\pi$. This set $\mathcal{L}$ is a Cameron-Liebler set in $\mathcal{P}$, but not a degree one Cameron-Liebler set. This follows from Theorem \ref{stelling} with $i=1$: The number of generators of $\mathcal{L}$, meeting a plane $\alpha$ of $\mathcal{L}$ in a line depends on  whether $\alpha$ equals $\pi$ or not. As those two numbers, for $\alpha = \pi$ and $\alpha \neq \pi$ are different, the property in Theorem \ref{stelling} does not hold. This implies that the set $\mathcal{L}$ is no degree one Cameron-Liebler set. By similar arguments, we can also use Theorem \ref{stelling} with $i=2$, to show that a base-plane is not a degree one Cameron-Liebler set. However, the equalities for $i=3$ in Theorem \ref{stelling} hold. 
\end{remark}

\begin{remark}\label{vb1}

A hyperbolic class is the set of all generators of one class of a hyperbolic quadric $Q^+(4n+1,q)$ embedded in a polar space $\mathcal{P}$ with $\mathcal{P}=Q(4n+2,q)$ or $\mathcal{P}=W(4n+1,q)$, $q$ even. We know that this set is a Cameron-Liebler set, see \cite[Remark 3.25]{CLpolar}, but we can prove that this set is not a degree one Cameron-Liebler set, by considering $\im(B^T)$, where $B$ is the incidence matrix of hyperbolic classes and generators. Every hyperbolic class corresponds to a row in the matrix $B$. If the characteristic vectors of all hyperbolic classes would lie in $V_0 \perp V_1$, then $\im(B^T)\subseteq V_0\perp V_1$. This gives a contradiction since $\im(B^T)= V_0\perp V_1\perp V_d$ by \cite[Lemma 3.26]{CLpolar}.\\
Remark that for the polar spaces $W(4n+1,q)$, $q$ odd, we do not have this example as there is no hyperbolic quadric $Q^+(4n+1,q)$ embedded in these symplectic polar spaces.
\end{remark}

In the previous remark we found that one class of a hyperbolic quadric $Q^+(4n+1,q)$ embedded in a $Q(4n+2,q)$ or $W(4n+1,q)$, $q$ even is no degree one Cameron-Liebler set. In the next example we show that an embedded hyperbolic quadric (so both classes) is a degree one Cameron-Liebler set in the polar spaces $Q(4n+2,q)$ and $W(4n+1,q)$, $q$ even. 

\begin{example}[{\cite[Example $4.4$]{CLpolar}}] \label{vb}
Consider a polar space $\mathcal{P}$ as in Remark \ref{vb1}. 
By Lemma \ref{lemmaferdi} we know
that this set of generators is a degree one Cameron-Liebler set, hence also a Cameron-Liebler set.
\end{example}
\begin{table}[h]\begin{center}
  \begin{tabular}{ | l |c|c| }
    \hline
    Example & CL & degree one CL \\ \hline \hline
    All generators of $\mathcal{P}$ &$\times$& $\times$ \\ \hline
     Point-pencil (defined in Section \ref{setcionfcps} ). & $\times$& $\times$ \\ \hline
     Base-plane (defined in Example \ref{commentvb4.6}). & $\times$ & \\ \hline
    Hyperbolic class (defined in Comment \ref{vb1}). & $\times$& \\ \hline 
    Embedded hyperbolic quadric (defined in Example \ref{vb}). &$\times$& $\times$  \\ \hline
    
  \end{tabular}
  \caption{Examples of Cameron-Liebler and degree one Cameron-Liebler sets.}\label{tabelexample}
\end{center}\end{table}
Remark by Lemma \ref{basislemma4}(3), that also the complements of the sets in Table \ref{tabelexample} are Cameron-Liebler sets or degree one Cameron-Liebler sets respectively.

\section{Polar spaces $Q^+(2d-1,q)$, $d$ even}
In the previous section we introduced degree one Cameron-Liebler sets while in this section we handle Cameron-Liebler sets in polar spaces. We focus on Cameron-Liebler in one class of generators in the polar spaces $Q^+(2d-1,q)$, $d$ even. These Cameron-Liebler sets were introduced in \cite[Section 3]{CLpolar} and are defined in only one class of generators, in contrast to the (degree one) Cameron-Liebler sets in other polar spaces.

It is known that the generators of a hyperbolic quadric $Q^+(2d-1,q)$ can be divided in two classes such that for any two generators $\pi$ and $\pi'$ we have $\dim(\pi\cap\pi')\equiv d\pmod{2}$ iff $\pi$ and $\pi'$ belong to the same class. By restricting the classical association scheme of the hyperbolic quadric $Q^+(2d-1,q)$ to the even relations, we define an association scheme for one class of generators. For more information, see \cite[Remark $2.18$, Lemma $3.12$]{CLpolar}. 
Let $R'_i$ and $A'_i$ be $R_{2i}$ and $A_{2i}$ respectively, restricted to the rows and columns corresponding to the generators of this class. Let  $V'_j$ be $ V_j \perp V_{d-j}$, also restricted to the subspace corresponding to these generators.

For the polar spaces $Q^+(2d-1,q)$, $d$ even, we thus have the relations $R_i'$, $i = 0, \dots ,\frac{d}{2}$, and the eigenspaces $V_j'$, $j = 0, \dots, \frac{d}{2}$. For this association scheme on one class of generators we give the analogue of Lemma \ref{lemma2}.

\begin{lemma} \label{lemma4}
The eigenvalue $P_{1,2i}$ of $A_i' = A_{2i}$ corresponds only with the eigenspace $V_1' = V_1 \perp V_{d-1}$ for the classical polar spaces $Q^+(2d-1,q)$, $d$ even.
\end{lemma}
\begin{proof}
This lemma follows from Lemma \ref{lemma2} as for the hyperbolic quadrics $Q^+(2d-1,q)$ we found that $P_{1k}=P_{d-1,k}$ for $j$ even. This implies that the eigenvalue $P_{1,2i}$ corresponds with $V_1 \perp V_{d-1}$.
\end{proof}

Here again, we find a new equivalent definition.

\begin{theorem} \label{stelling3}
Let $\mathcal{G}$ be a class of generators of the hyperbolic quadric $Q^+(2d-1,q)$ of even rank $d$ and let $\mathcal{L}$ be a set of generators of $\mathcal{G}$. The set $\mathcal{L}$ is a Cameron-Liebler set of generators in $\mathcal{G}$ if and only if for every generator $\pi$ in $\mathcal{G}$, the number of elements of $\mathcal{L}$ meeting $\pi$ in a $(d-2i-1)$-space equals \\
\begin{align*}
\left\{ \begin{matrix}
 \left( (x-1) \begin{bmatrix}d-1 \\2i-1\end{bmatrix} +q^{2i-1}\begin{bmatrix}d-1 \\ 2i\end{bmatrix} \right) q^{(2i-1)(i-1)} & \mbox{If } \pi \in \mathcal{L}\\  x \begin{bmatrix} d-1 \\ 2i-1\end{bmatrix} q^{(2i-1)(i-1)} & \mbox{If }\pi \notin \mathcal{L}. 
 \end{matrix}\right.
\end{align*}
\end{theorem}
\begin{proof}
Let $\mathcal{L}$ be a set of generators in $\mathcal{G}$ with the property described in the theorem, then the first implication is a direct application of Theorem \ref{stellingalgemeen} with $\Omega$ the set of all generators in $\mathcal{G}$, $R_i$ the relation $R'_i=\{(\pi, \pi')|\dim(\pi \cap\pi') = d-2i-1\}$, and \begin{align*}
 \alpha_i &= \left( (x-1) \begin{bmatrix}d-1 \\2i-1\end{bmatrix} +q^{2i-1}\begin{bmatrix}d-1 \\ 2i\end{bmatrix} \right) q^{(2i-1)(i-1)},\\
 \beta_i &= x \begin{bmatrix} d-1 \\ 2i-1\end{bmatrix} q^{(2i-1)(i-1)}.\end{align*}
As $\alpha_i - \beta_i = P_{1,2i}$, we find that $v_i = \chi + \frac{\beta_i}{P_{1,2i} - P_{0,2i}}\textbf{\textit{j}} \in V'_1$, hence $\chi \in V'_0 \perp V'_1 $ and by Lemma $3.15$ in \cite{CLpolar} we know that $\chi \in$ $\im(A^T)$. 
 Now it follows from \cite[Definition 3.16(iv)]{CLpolar} that $\mathcal{L}$ is a (degree one) Cameron-Liebler set of $\mathcal{G}$. The other implication is  \cite[Lemma $4.10$]{CLpolar}.\end{proof}\\

\section{Classification results}
We try to use the ideas from the classification results for Cameron-Liebler sets of polar spaces of type $I$ and the polar spaces $Q^+(2d-1,q)$, $d$ even, in \cite[Section 6]{CLpolar}, to find classification results for degree one Cameron-Liebler sets in polar spaces.\\
We start with some definitions and a lemma that proves that the parameter $x$ is always an integer.
\begin{definition}
A \emph{partial ovoid} is a set of points in a polar space such that each generator contains at most one
point of this set. It is called an \emph{ovoid} if each generator contains precisely one point of the set.
\end{definition}
\begin{definition}
An Erd\H{o}s-Ko-Rado (EKR) set of $k$-spaces is a set of k-spaces which
are pairwise not disjoint.
\end{definition}
\begin{lemma}
If $\mathcal{L}$ is a degree one Cameron-Liebler set in a polar space $\mathcal{P}$ with parameter $x$, then $x\in \mathbb{N}$
\end{lemma}
\begin{proof}
For all polar spaces, except the hyperbolic quadrics $Q^+(2d-1,q)$, $d$ even, we refer to \cite[Lemma 4.8]{CLpolar}.

Suppose now that $\mathcal{L}$ is a degree one Cameron-Liebler set in $\mathcal{P}=Q^+(2d-1,q)$, $d$ even, with parameter $x$. Then $\mathcal{L}$ is also a Cameron-Liebler set in $\mathcal{P}$ with parameter $x$. If $\Omega_1$ and $\Omega_2$ are the two classes of generators in $\mathcal{P}$, then $\mathcal{L}\cap \Omega_1$ and $\mathcal{L}\cap \Omega_2$ are Cameron-Liebler sets of $\Omega_1$ and $\Omega_2$ with parameter $x$, by \cite[Theorem 3.20]{CLpolar}. Hence $x$ is the parameter of a Cameron-Liebler set in one class of generators of $Q^+(2d-1,q)$, $d$ even. This implies, by \cite[Lemma 4.8]{CLpolar}, that $x\in \mathbb{N}$. \end{proof}\\

Now we continue with a classification result for degree one Cameron-Liebler sets with parameter $1$ in all polar spaces.
\begin{theorem}
A degree one Cameron-Liebler set in a polar space $\mathcal{P}$ of rank $d$ with parameter $1$ is a point-pencil.
\end{theorem}
\begin{proof}
For the polar spaces of type $I$ and $III$, the theorem follows from \cite[Theorem 6.4]{CLpolar} as any degree one Cameron-Liebler set is a Cameron-Liebler set and since  a base-plane, a base-solid and a hyperbolic class, are  no degree one Cameron-Liebler sets (see Remark \ref{commentvb4.6} and Remark \ref{vb1}).
The theorem for the polar spaces of type $I$ follows from \cite[Theorem 6.4]{CLpolar} as degree one Cameron-Liebler sets and Cameron-Liebler sets coincide for these polar spaces. 
Let $\mathcal{L}$ be a degree one Cameron-Liebler set with parameter $1$ in a polar space $\mathcal{P}$ and remark that $\mathcal{L}$ is an EKR set with size $\prod_{i=0}^{d-2}(q^{i+e}+1)$. For the polar spaces $Q(4n+2,q)$ and $W(4n+1,q)$, we find by \cite[Theorem 23, Theorem 40]{pepe} that $\mathcal{L}$ is a point-pencil, a hyperbolic class or a base-plane if $n=1$. Using Remark \ref{vb1} and Remark \ref{commentvb4.6}, we find that the last two possibilities are no degree one Cameron-Liebler sets. Hence $\mathcal{L}$ is a point-pencil.

Suppose now that $\mathcal{P}$ is the hyperbolic quadric $Q^+(4n-1,q)$ with $\Omega_1$ and $\Omega_2$ the two classes of generators. By \cite[Theorem 3.20]{CLpolar}, we know that $\mathcal{L}\cap \Omega_1$ and $\mathcal{L}\cap \Omega_2$ are Cameron-Liebler sets in $\Omega_1, \Omega_2$ respectively,  with parameter $1$. Using \cite[Theorem 6.4]{CLpolar} we see that $\mathcal{L}\cap \Omega_i$ is a point-pencil or a base-solid if $n=2$ for $i=1,2$. A base-solid is the set of all $3$-spaces intersecting a fixed 3-space (the base) in precisely a plane. Note that all elements of the base-solid belong to a different class of the hyperbolic quadric than the base itself.  


If $n=2$, so $d=4$, and $\mathcal{L}\cap \Omega_1$ or $\mathcal{L}\cap \Omega_2$ is a base-solid with base $\pi$, then there are at least $(q+1)(q^2+1)$ elements of $\mathcal{L}$ meeting $\pi$ in a plane. This contradicts Theorem \ref{stelling}, whether $\pi\in \mathcal{L}$ or not. So we find that $\mathcal{L}\cap \Omega_1$ and $\mathcal{L}\cap \Omega_2$ are both point-pencils with vertex $v_1$ and $v_2$ respectively. Now we show that $v_1=v_2$. Suppose $v_1\neq v_2$. Consider a generator $\alpha \in \Omega_2\setminus \mathcal{L}$ through $v_1$. Then $\alpha$ intersects $q^2+q+1$ generators of $\mathcal{L}\cap \Omega_1$ in a plane through $v_1$. This gives a contradiction with Theorem \ref{stelling}, which proves that $v_1=v_2$. Hence $\mathcal{L}$ is a point-pencil through $v_1=v_2$. \end{proof} \\

The classification result in \cite[Theorem 6.7]{CLpolar} for polar spaces of type $I$ is also valid for degree one Cameron-Liebler sets in all polar spaces.
\begin{theorem}
Let $\mathcal{P}$ be a finite classical polar space of rank $d$ and parameter $e$, and let $\mathcal{L}$ be a degree one Cameron-Liebler set of $\mathcal{P}$ with parameter $x$. If $x\leq q^{e-1}+1$ then $\mathcal{L}$ is the union of $x$ point-pencils whose vertices are pairwise non-collinear or $x=q^{e-1}+1$ and  $\mathcal{L}$ is the set of generators in an embedded polar space of rank $d$ and with parameter $e-1$.
\end{theorem}
\begin{proof}
In Lemma $6.5$, Theorem $6.6$ and Theorem $6.7$ of \cite{CLpolar}, the authors use \cite[Lemma $4.9$]{CLpolar} to prove the classification result. We can use the same proof as we can use Theorem \ref{stelling} instead of \cite[Lemma $4.9$]{CLpolar}. \end{proof}\\

Remark that the last possibility corresponds to an embedded hyperbolic quadric $Q^+(4n+1,q)$ if $\mathcal{P}=Q(4n+2,q)$ or $\mathcal{P}=W(4n+1,q)$ with $q$ even. If $\mathcal{P}=W(4n+1,q)$ with $q$ odd, then $\mathcal{P}$ admits no embedded polar space with rank $n$ and parameter $e-1=0$.\\

For the symplectic polar space $W(5,q)$ and the parabolic quadric $Q(6,q)$ we give a stronger classification result.
Remark the polar spaces $Q(6,q)$ and $W(5,q)$ are isomorphic for $q$ even.
We find $W(5,q)$, for $q$ even, by a projection of $Q(6,q)$ from the nucleus $N$ of $Q(6,q)$ to a hyperplane not through $N$ in the ambient projective space $\PG(6,q)$. In this way, there is a one-one connection between the planes of $W(5,q)$ and the planes of $Q(6,q)$.
We start with some lemmas. 

\begin{lemma} \label{lemmas1s2d1d2}
Let $\mathcal{L}$ be a degree one Cameron-Liebler set of generators (planes) in $W(5,q)$ or $Q(6,q)$ with parameter $x$. 
\begin{enumerate}
\item For every $\pi \in \mathcal{L}$, there are $s_1$ elements of $\mathcal{L}$ meeting $\pi$. 
\item For skew $\pi, \pi'\in \mathcal{L}$, there exist exactly $d_2$ subspaces in $\mathcal{L}$ that are skew to both $\pi$ and $\pi'$ and there exist $s_2$ subspaces in $\mathcal{L}$ that meet both $\pi$ and $\pi'$.

Here, $d_2$, $s_1$ and $s_2$ are given by:
\begin{align*}
d_2(q,x) &= (x-2)q^2(q-1)\\
s_1(q,x) &= x(q^2+1)(q+1)-(x-1)q^3 = q^3+x(q^2+q+1)\\
s_2(q,x) &= x(q^2+1)(q+1)-2(x-1)q^3 +d_2(q,x).
\end{align*}
\end{enumerate}
\end{lemma}
\begin{proof}Let $\mathcal{P}$ be the polar space $W(5,q)$ or $Q(6,q)$, hence $d=3$ and $e=1$.
\begin{enumerate}
\item This follows directly from Theorem \ref{stelling}, for $i=d$ and $|\mathcal{L}|=x(q^2+1)(q+1)$. 
\item  Let
$\chi_\pi$ and $\chi_{\pi'}$ be the characteristic vectors of $\{\pi\}$ and $\{\pi'\}$, respectively. Let $\mathcal{Z}$ be the set of all planes in $\mathcal{P}$ disjoint to $\pi$ and $\pi'$, and let $\chi_\mathcal{Z}$ be its characteristic vector. Furthermore, let $v_\pi$ and $v_{\pi'}$ be the incidence vectors of $\pi$ and $\pi'$, respectively, with their positions corresponding to the points of $\mathcal{P}$. Note that $A\chi_\pi = v_\pi$ and $A\chi_{\pi'} = v_{\pi'}$. \\
The number of planes through a point $P\notin \pi \cup \pi'$ and disjoint to $\pi$ and $\pi'$ is the number of lines in $P^\perp$, disjoint to the lines corresponding to $\pi$ and $\pi'$. By \cite[Corollary $19$]{kms} this number equals $q^2(q-1)$, and we find: 
\begin{align*}
A\chi_\mathcal{Z} &=q^2(q-1)(\textbf{\textit{j}}-v_\pi-v_{\pi'}) \\
&=q^2(q-1)\left(A\frac{\textbf{\textit{j}}}{(q^2+1)(q+1)}-A\chi_\pi-A\chi_{\pi'}\right)\\
\Leftrightarrow\qquad&\chi_\mathcal{Z}-q^2(q-1)\left(\frac{\textbf{\textit{j}}}{(q^2+1)(q+1)}-\chi_\pi-\chi_{\pi'}\right) \in \ker(A).
\end{align*}
We know that the characteristic vector $\chi$ of $\mathcal{L}$ is included in $\ker(A)^\perp$. This implies:
\begin{align*}
&&\chi_\mathcal{Z} \cdot \chi &=q^2(q-1)\left(\frac{\textbf{\textit{j}}\cdot \chi}{(q^2+1)(q+1)}-\chi(\pi)-\chi(\pi')\right) \\
&\Leftrightarrow & |\mathcal{Z}\cap \mathcal{L}|  &=(x-2)q^2(q-1) 
\end{align*}
which gives the formula for $d_2(q,x)$.  The formula for $s_2(q,x)$ follows from the inclusion-exclusion principle.  \qedhere
\end{enumerate} \end{proof} 
In the following lemma, corollary and theorem, we will use $s_1,s_2,d_1,d_2$ for the values $s_1(q,x),s_2(q,x),$ $d_1(q,x), d_2(q,x)$ if the field size $q$ and the parameter $x$ are clear from the context. For the definition of these values, we refer to the previous lemma. 

The following lemma is a generalization of Lemma $2.4$ in \cite{Klaus}.
\begin{lemma}\label{lemmaklaus}
If $c$ is a nonnegative integer such that 
\begin{align*}
(c+1)s_1-\binom{c+1}{2}s_2 > x(q^2+1)(q+1)\;,
\end{align*}
then no degree one Cameron-Liebler set of generators in $W(5,q)$ or $Q(6,q)$ with parameter $x$ contains $c+1$ mutually skew generators. 
\end{lemma}
\begin{proof}
Let $\mathcal{P}$ be the polar space $W(5,q)$ or $Q(6,q)$ and assume that $\mathcal{P}$ has a degree one Cameron-Liebler set $\mathcal{L}$ of generators with parameter $x$ that contains $c+1$ mutually disjoint subspaces $\pi_0,\pi_1,\dots,\pi_c$. Lemma \ref{lemmas1s2d1d2} shows that $\pi_i$, meets at least $s_1(q,x)-i\cdot s_2(q,x)$ elements of $\mathcal{L}$ that are skew to $\pi_0, \pi_1, \dots,\pi_{i-1}$. Hence $x(q^2+1)(q+1) = |\mathcal{L}| \geq (c+1) s_1-\sum_{i=0}^c i s_2$ which contradicts the assumption.
\end{proof}

\begin{gevolg}\label{gevolgklaus}
A degree one Cameron-Liebler set of generators in $W(5,q)$ or $Q(6,q)$ with parameter $2\leq x\leq \sqrt[3]{2q^2}-\frac{\sqrt[3]{4q}}{3}+\frac{1}{6}$ contains at most $x$ pairwise disjoint generators.
\end{gevolg}
\begin{proof}
Let $\mathcal{L}$ be a degree one Cameron-Liebler set of generators in $W(5,q)$ or $Q(6,q)$ with parameter $x$. Using Lemma \ref{lemmaklaus} for $e=1,d=3, c=x$ we find that if $q^3-q^2x+\frac{q+1}{2}x^2-\frac{q+1}{2}x^3>0$, then $\mathcal{L}$ contains at most $x$ pairwise disjoint generators. Since  $f_q(x)=q^3-q^2x-\frac{q+1}{2}x^2(x-1)$ is decreasing on $[1,+\infty[$, we find that it is sufficient that $f_q\left(\sqrt[3]{2q^2}-\frac{\sqrt[3]{4q}}{3}+\frac{1}{6}\right)>0$, as we only consider the values of $x$ in $[2,\dots,\sqrt[3]{2q^2}-\frac{\sqrt[3]{4q}}{3}+\frac{1}{6}]$. Using a computer algebra packet, it can be checked that $f_q\left(\sqrt[3]{2q^2}-\frac{\sqrt[3]{4q}}{3}+\frac{1}{6}\right)>0$ for all $q\geq 2$.
\end{proof}

\begin{theorem}
A degree one Cameron-Liebler set $\mathcal{L}$ of generators in $W(5,q)$ or $Q(6,q)$ with parameter $2\leq x\leq \sqrt[3]{2q^2}-\frac{\sqrt[3]{4q}}{3}+\frac{1}{6}$ is the union of $\alpha$ embedded disjoint hyperbolic quadrics $Q^+(5,q)$ and $x-2\alpha$ point-pencils whose vertices are pairwise disjoint and not contained in the $\alpha$ hyperbolic quadrics $Q^+(5,q)$. For the polar space $Q(6,q)$ or $W(5,q)$ with $q$ even, $\alpha\in \{0,...,\lfloor \frac{x}{2} \rfloor\}$, for the polar space $W(5,q)$ with $q$ odd, $\alpha=0$.
\end{theorem}
\begin{proof}
Let $\mathcal{P}$ be the polar space $W(5,q)$ or $Q(6,q)$ and $\mathcal{L}$ be a degree one Cameron-Liebler set in $\mathcal{P}$. Note that the generators in these polar spaces are planes.  By Corollary \ref{gevolgklaus}, there are $c$ pairwise disjoint planes $\pi_1,\pi_2,\dots, \pi_c$ with $c\leq x$ in $\mathcal{L}$. Let $S_i$ be the set of planes in $\mathcal{L}$ intersecting $\pi_i$ and not intersecting $\pi_j$ for all $j\neq i$. By Lemma \ref{lemmas1s2d1d2} there are, for a fixed $i$, at least $s_1-(c-1)s_2\geq s_1-(x-1)s_2 =q^3-(x-2)q^2-(x^2-2x)(q+1)$ planes in $S_i$. As $S_i$ is an EKR set by Corollary \ref{gevolgklaus}, $S_i$ has to be a part of a point-pencil (PP), a base plane (BP) or one class of an embedded hyperbolic quadric $Q^+(5,q)$ (CEHQ). Remark that if $\mathcal{P}$ is $W(5,q)$ with $q$ odd, then $\mathcal{P}$ cannot contain a CEHQ, so for this polar space, the only possibilities are a PP of BP,
by \cite[Lemmas $3.3.7$, $3.3.8$ and $3.3.16$]{phdmaarten}. Using Theorem \ref{stelling}, we can prove that if the set $S_i$ is a part of a PP, BP or CEHQ, then $\mathcal{L}$ has to contain all planes of this PP, BP or CEHQ. We show this for the case where the set of planes forms a part of a PP.  So assume $S_i$ is a subset of the point-pencil with vertex $P$, and there is a plane $\gamma \notin \mathcal{L}$ through $P$.  This would imply that $\gamma$ meets at least $q^3-(x-2)q^2-(x^2-2x)(q+1)$ planes in $\mathcal{L}$ non-trivially. This gives a contradiction by Theorem \ref{stelling} for $i=1$ and $i=2$, as $\gamma \notin \mathcal{L}$ intersects with precisely $x(q^2+q+1)<q^3-(x-2)q^2-(x^2-2x)(q+1)$ planes of $\mathcal{L}$.\\
This argument also works for the BP and CEHQ, so we can conclude that if $\mathcal{L}$ contains an $S_i$ which is  a part of a PP, BP or CEHQ, then $\mathcal{L}$ has to contain the whole PP, BP or CEHQ respectively, which we will call $\mathcal{L}_i$. \\

Remark first that $\mathcal{L}$ cannot contain a BP with base $\pi$ as then $\pi \in \mathcal{L}$ intersects $q^3+q^2+q>q^2+q+x-1$ planes of $\mathcal{L}$ in a line, which gives a contradiction with Theorem \ref{stelling}. This implies that all sets $\mathcal{L}_i$ are PP's or CEHQ's. Now we show that every two sets $\mathcal{L}_i$ and $\mathcal{L}_j$ are disjoint. Suppose first that $\mathcal{L}_i$ and $\mathcal{L}_j$ are two PP's with vertices $P_i$ and $P_j$ respectively, that are not disjoint. Then there are at most $q+1$ planes in $\mathcal{L}_i \cap \mathcal{L}_j$ and let $\alpha$ be one of them. Now we see that $\alpha$ meets at least $2(q^3+q^2+q+1)-(q+1)$ elements of $\mathcal{L}$, contradicting Theorem \ref{stelling}. 
If $\mathcal{L}_i$ and $\mathcal{L}_j$ are two non-disjoint CEHQ's or a CEHQ and a PP that are non-disjoint, then we can use the same arguments as above: In both cases, there are at most $q+1$ planes in $\mathcal{L}_i \cap \mathcal{L}_j$, which implies that a plane $\alpha \in \mathcal{L}_i \cap \mathcal{L}_j$ meets at least $2(q^3+q^2+q+1)-(q+1)$ elements of $\mathcal{L}$ non-trivially, contradicting Theorem \ref{stelling}.


Now we know that $\mathcal{L}$ contains the disjoint union of $c\leq x$ sets $\mathcal{L}_i$ of planes, where every set is a PP or CEHQ. As the number of planes in a PP or CEHQ equals $(q^2+1)(q+1)$, and the total number of planes in $\mathcal{L}$ equals $x(q^2+1)(q+1)$ (see Lemma \ref{basislemma4}(2)), we see that $\mathcal{L}$ equals the disjoint union of $x$ sets $\mathcal{L}_i$. \\
To end this proof, we want to show that the only possible composition of $\mathcal{L}$ exists of PP's and embedded hyperbolic quadrics.  If $\mathcal{L}$ contains one class of an embedded hyperbolic quadric, then $\mathcal{L}$ also contains the other class of this hyperbolic quadric. This also follows from Theorem \ref{stelling}: suppose $\mathcal{L}$ contains only one class of an embedded hyperbolic quadric and let $\pi$ be a plane of the other class of this embedded hyperbolic quadric. Then we can show that $\pi$ is also a plane of $\mathcal{L}$: we know that $\pi$ meets $q^2+q+1$ planes of the hyperbolic quadric in a line, so at least so many planes of $\mathcal{L}$, in a line. But if $\pi \notin \mathcal{L}$, then by Theorem \ref{stelling}, $\pi$ can only meet $x< \sqrt[3]{2q^2}$ planes of $\mathcal{L}$ in a line, a contradiction.

This implies that $\mathcal{L}$ has to be the disjoint union of point-pencils and embedded hyperbolic quadrics. 
Remark that two point-pencils are disjoint if the corresponding vertices are non-collinear. As there exists a partial ovoid of size $q+1$ in $\mathcal{P}$, we can find $x$ pairwise disjoint point-pencils.\\
Note that for $q$ odd and $\mathcal{P} = W(5,q)$, there are no embedded hyperbolic quadrics, so in this case $\mathcal{L}$ is the disjoint union of $x$ point-pencils. 
We end the proof by showing that, for $\mathcal{P} = Q(6,q)$ or $\mathcal{P} = W(5,q)$ and $q$ even, there exist disjoint embedded hyperbolic quadrics in $\mathcal{P}$. 
It suffices to show this only for $\mathcal{P} = Q(6,q)$, by the connection between $Q(6,q)$ and $W(5,q)$ for $q$ even.
Consider two embedded hyperbolic quadrics $Q^+(5,q)$ in $Q(6,q)$, that intersect in a parabolic quadric $Q(4,q)$. These two hyperbolic quadrics have no planes in common as the generators of $Q(4,q)$ are lines, so these two embedded hyperbolic quadrics are disjoint. 
Remark that the disjoint union of point-pencils and the disjoint union of embedded hyperbolic quadrics is a degree one Cameron-Liebler set by Lemma \ref{basislemma4}$(4)$, as a point-pencil is a degree one Cameron-Liebler set and for $\mathcal{P}\neq W(5,q)$ or $q$ even, an embedded hyperbolic quadric of the same rank is also a degree one Cameron-Liebler set. 
\end{proof}

This theorem agrees with Conjecture $5.1.3$ in \cite{Ferdinand.}, as this conjecture says that every degree one Cameron-Liebler set in $W(5,q)$ is the disjoint union of non-degenerate hyperplane sections and point-pencils. 
\begin{remark}
Recall that the disjoint union of point-pencils and disjoint embedded hyperbolic quadrics is also an example of a degree one Cameron-Liebler set of generators in the other polar spaces of type $III$ (see Lemma \ref{basislemma4} and Example \ref{vb}).

We also remark that we could not generalize this classification result to other classical polar spaces, as for these polar spaces, there is not enough information known about large EKR sets in these polar spaces.  For the polar spaces $Q^+(4n+1,q)$ there are some EKR results in \cite{maarten2}. Since in this case, the large examples of EKR sets have much more elements than the largest known Cameron-Liebler sets, we cannot use these results.
\end{remark}

\section*{Summary of properties for type \texorpdfstring{$III$}{III}}
In Table \ref{tabeldef} we give an overview of properties where we distinguish sufficient properties, necessary properties and characteristic properties, or definitions for Cameron-Liebler sets and for degree one Cameron-Liebler sets. 

Suppose in this table that $\mathcal{L}$ is a set of generators in the polar space $\mathcal{P}$ of type $III$, with characteristic vector $\chi$. Suppose also that $\pi$ is a generator in $\mathcal{P}$, not necessarily in $\mathcal{L}$.\\
\begin{table}[h]\begin{center}
  \begin{tabular}{ | l l |c|c| }
    \hline
    Property &&  CL & degree one CL \\ \hline \hline
     $\chi \in V_0 \perp V_1$. && $S$& $C$ \\ \hline
    $\forall \pi \in \mathcal{P}$: $|\{\tau \in \mathcal{L}| \tau \cap \pi = \emptyset\}| = (x-\chi(\pi))q^{\binom d 2}$. && $C$& $N$ \\ \hline 
    $\chi-\frac{x}{q^d+1}\textbf{\textit{j}}$ is an eigenvector of $A_d$ with eigenvalue $-q^{\binom d 2}$. && $C$ & $N$\\ \hline
    $\forall \pi \in \mathcal{P}, |\{\tau \in \mathcal{L}| \dim(\tau \cap \pi) = d-i-1 \}|=$ (\ref{formulelang}),  for $0\leq i <d$ &&$S$& $C$ \\ \hline
 
    If $\mathcal{P}$ admits a spread, then $|\mathcal{L}\cap S| = x$ $\forall$ spread $S\in \mathcal{P}$. && $C$ &$N$ \\ \hline

  \end{tabular}
  \caption{Overview of the sufficient ($S$), necessary ($N$) and characterising ($C$) properties.}\label{tabeldef}
\end{center}\end{table}

\subsection*{Acknowledgements}
The research of Jozefien D'haeseleer is supported by the FWO (Research Foundation Flanders). 
The authors thank Leo Storme for his suggestions while writing this article.

\end{document}